\documentclass{amsart}
\usepackage{graphicx} 

\usepackage{amsmath}
\usepackage{amssymb}
\usepackage{amsthm}

\usepackage{dsfont}

\usepackage{esint}

\usepackage{mathrsfs}

\theoremstyle{theorem}
\newtheorem{theorem}{Theorem}[section]

\theoremstyle{definition}

\theoremstyle{definition}

\theoremstyle{definition}

\theoremstyle{theorem}
\newtheorem{lemma}[theorem]{Lemma}

\theoremstyle{theorem}
\newtheorem{proposition}[theorem]{Proposition}

\theoremstyle{theorem}

\theoremstyle{corollary}
\newtheorem{corollary}[theorem]{Corollary}

\theoremstyle{question}
\newtheorem{question}[theorem]{Question}

\theoremstyle{remark}
\newtheorem{remark}[theorem]{Remark}

\DeclareMathOperator{\Ker}{Ker}

\usepackage{polynom}

\DeclareMathOperator{\Aut}{Aut}

\DeclareMathOperator{\Sym}{Sym}

\DeclareMathOperator{\Mod}{Mod}
\DeclareMathOperator{\Hol}{Hol}
\DeclareMathOperator{\Diff}{Diff}

\usepackage{comment}
\usepackage{appendix}

\usepackage{tikz-cd}
\usetikzlibrary{matrix,arrows,decorations.pathmorphing}

\title{Entropy-Minimizing Diffeomorphisms on a $G_2$-Manifold}
\author{Ollie Thakar}
\address{\parbox{\linewidth}{Department of Mathematics, Harvard University, Massachusetts, 02138}}
\email{othakar@math.harvard.edu}

\begin{document}

\maketitle

\begin{abstract}
    In this paper, we construct infinitely many diffeomorphisms of a Joyce manifold $M$ which achieve Yomdin's homological lower bound for topological entropy, imitating a recent construction of Farb-Looijenga for K3 surfaces. Moreover, following a recent paper by Crowley-Goette-Hertl, we show these diffeomorphisms act freely on a connected component of the Teichm\"uller space of $G_2$ structures on $M$, and hence that the homotopy moduli space of $G_2$ structures on $M$ has infinite fundamental group. We also discuss a putative analogy between dynamics on a $G_2$ manifold and that of an algebraic surface, and prove a theorem about its limitations. 
\end{abstract}

\section{Introduction}

The compact Lie group $G_2$ is defined as the subgroup of $SO(7)$ which preserves the 3-form $\varphi_0 \in \Lambda^3\mathbb{R}^7$ defined as follows: $$\varphi_0 = e_1\wedge e_2\wedge e_5+ e_3\wedge e_4\wedge e_5+e_1\wedge e_3\wedge e_6-e_2\wedge e_4\wedge e_6+e_1\wedge e_4\wedge e_7-e_2\wedge e_3\wedge e_7+e_5\wedge e_6\wedge e_7,$$ where $e_1,\dots, e_7$ is an orthonormal basis of $\mathbb{R}^7.$ A 7-manifold $M$ has a \emph{torsion-free $G_2$-structure} or is a \emph{$G_2$-manifold} if it has a Riemannian metric $g$ such that the holonomy group $\Hol(g) \subseteq G_2$, or equivalently, it has a harmonic 3-form $\varphi\in\Omega^3_M$ which, at each point $p\in M$, can be expressed as $\varphi_0$ upon choosing an appropriate basis $e_1,\dots, e_7\in T_p^*M.$ A $G_2$ manifold is \emph{irreducible} if its holonomy group is precisely $G_2$, which when the manifold is closed is equivalent to having finite fundamental group \cite{Joyce1}.

The ``Teichm\"uller space'' $\mathcal{T}(M)$ of torsion-free $G_2$-structures divided by the diffeomorphisms of $M$ isotopic to the identity is known to be a smooth manifold of dimension $b^3(M)$ \cite[Theorem C]{Joyce1}. However, little is known about the moduli space $\mathcal{M}(M)$ of torsion-free $G_2$-structures divided by the full diffeomorphism group. Clearly, $\mathcal{M}(M)$ is the quotient of $\mathcal{T}(M)$ by the action of the smooth \emph{mapping class group} $\Mod(M):=\pi_0\Diff(M).$

This paper provides an example of an irreducible $G_2$-manifold that has infinite mapping class group explicitly visible by formulas in coordinates, and with diffeomorphisms satisfying an interesting dynamical property:

\begin{theorem}
    There exists a closed irreducible $G_2$-manifold $(M, \varphi)$ and an infinite family of diffeomorphisms $f_i$, each in a distinct mapping class, satisfying: \begin{enumerate}
        \item Each $f_i$ has positive topological entropy and minimizes topological entropy in its homotopy class.
        \item Each $f_i$ acts freely on the space $\mathcal{T}(M).$
        \item Each $f_i$ preserves the connected component of $\mathcal{T}(M)$ containing $[\varphi].$
    \end{enumerate}
\end{theorem}

\begin{proof}[Proof of (2).]
    Suppose $f_i$ does not act freely on $\mathcal{T}(M).$ Then, there exists a torsion-free $G_2$-structure $\widetilde\varphi$ and some diffeomorphism $h: M\to M$ isotopic to the identity such that $f_i^*\widetilde\varphi = h^*\widetilde\varphi.$ Thus, $h\circ f_i$ preserves $\widetilde\varphi.$ Any diffeomorphism preserving a $G_2$ structure also preserves its Riemannian metric (see, for example, \cite[Thereom 2.3.3]{Karigiannis}), and hence must have zero topological entropy. Since $f_i$ has positive topological entropy and $h\circ f_i$ is clearly isotopic to $f_i,$ this contradicts (1).
\end{proof}

\begin{remark}
    Diffeomorphisms in this paper are to be class $C^\infty.$
\end{remark}

\subsection{Context}

This paper generalizes constructions from two different areas: dynamical systems on smooth manifolds, and moduli of $G_2$-metrics. We briefly discuss the interest of this result in each area:

\subsubsection{Dynamical Systems} Our construction is very closely modeled on a construction of Farb-Loojienga for pseudo-Anosov-type diffeomorphisms on K3 surfaces \cite{FL}. K3 surfaces can also be considered as manifolds admitting Riemannian metrics of exceptional holonomy, in this case $\text{Sp}(1)\subset SO(4)$. Farb and Loojienga construct infinitely many positive topological entropy diffeomorphisms on a K3 surface that minimize topological entropy in their isotopy class but also do not preserve any complex structure (therefore, do not preserve a torsion-free $\text{Sp}(1)$-structure either.)

Besides the Farb-Looijenga examples, the only other known examples of positive entropy diffeomorphisms that minimize entropy in their isotopy classes are biholomorphic automorphisms of K\"ahler manifolds, certain Anosov diffeomorphisms of nilmanifolds, and certain diffeomorphisms of $S^3\times S^3$ \cite{FL, Gromov, Franks, Fried}. Our construction produces another example of this phenomenon.

\subsubsection{$G_2$ Moduli} Upon inspecting the construction of $f_i$ in Theorem 1.1, we find that the image of $\text{Mod}(M)$ in $\Aut(H^*(M; \mathbb{Z}))$ is infinite. Moreover, an infinite subgroup of the mapping class group preserves at least one connected component of the Teichm\"uller space $\mathcal{T}(M)$. This comprises, to the author's knowledge, the first example of a connected component of the moduli space $\mathcal{M}(M)$ of a $G_2$-manifold which is covered infinitely by at least one connected component of $\mathcal{T}(M).$

The first example of a connected component of $\mathcal{T}(M)$ (for different $G_2$-manifolds $M$) with non-trivial homotopy groups was furnished by Crowley-Goette-Hertl in \cite{CGH}, in which they used a families gluing construction to create non-trivial bundles over $\mathbb{CP}^1$ of $G_2$-metrics on several examples of Joyce manifolds. Each Joyce manifold is composed by gluing a finite quotient of a flat torus to some products of ALE manifolds with tori. The examples in \cite{CGH} are constructed by gluing trivial families on the flat torus to non-trivial families on the ALE parts. The proof of part (3) of Theorem 1 also uses the families gluing theorem in \cite{CGH}, but instead involves gluing a trivial family on the ALE parts to a non-trivial family on the flat torus part. 

The argument in \cite{CGH} leads to considering ``homotopy moduli spaces'' of $G_2$ metrics defined in terms of homotopy quotients; the main result of this paper can also be thought of in this language. The section \cite[before Corollary E]{CGH} defines the \emph{homotopy moduli space} of a $G_2$ manifold $M$ to be the homotopy quotient: $$ \mathcal{HM}(M) := \{\text{metrics~with~holonomy~}G_2\} // \Diff(M),$$ which clearly has the \emph{homotopy Teichm\"uller space}: $$\mathcal{HT}(M):= \{\text{metrics~with~holonomy~}G_2\} // \Diff_0(M)$$ as a covering space. (The convention of \cite{CGH} is to call $\mathcal{HT}$ the homotopy moduli space.) We have the following corollary, graciously pointed out to me by Thorsten Hertl:

\begin{corollary}\label{Hertl}
    For the manifold $M$ considered in Theorem 1.1, the homotopy moduli space of $G_2$ structures has at least 1 component with infinite fundamental group.
\end{corollary}

\begin{proof}
    In the proof of (1) from Theorem 1.1, we will see that there exists at least one $f_i$ in the statement of the theorem with infinite order in the mapping class group (it will, in fact have infinite order in $\Aut(H^*(M; \mathbb{Z}))$.) From part (3) of Theorem 1.1, such $f_i$ preserves a connected component of $\mathcal{T}(M),$ so it therefore must preserve the corresponding component of $\mathcal{HT}(M).$ Hence, the restriction of the covering $\mathcal{HT}(M)\to\mathcal{HM}(M)$ to the connected components containing $[\varphi]$ has infinite deck transformation group.
\end{proof}

\subsection{Outline of the paper} Section 2 will be devoted to the construction of our manifold $M$ and the diffeomorphisms $f_i,$ while Section 3 will be devoted to the proof of Part (1) of Theorem 1.1, most of which consists of computing the topological entropy of the maps $f_i.$ In Section 4, we will prove Part (3) of Theorem 1.1, which has been restated as a result about gluing in families as Theorem 4.3. Finally, in Section 5, we will discuss Theorem 1 in the broader context of dynamical systems on $G_2$-manifolds.  

\subsection{Acknowledgments} The author would like to thank his advisor Peter Kronheimer for his support, as well as Seraphina Lee, Benson Farb, and Thorsten Hertl for valuable discussions. This paper was made with the support of Simons Foundation Award $\#$994330, Simons Collaboration on New Structures
in Low-Dimensional Topology.

\section{Farb-Looijenga Construction for a Joyce Manifold}

The manifold $M$ that we will use in Theorem 1 was first constructed by Joyce as \cite[Example 7]{Joyce2}. We will review the construction of $M$ as a smooth manifold now, reserving the review of Joyce's construction of a $G_2$ metric on $M$ for Section 4. Let $\xi = e^{2\pi i/3}$ and define the lattice $\Lambda = \mathbb{Z}^3\oplus\xi\mathbb{Z}^3\subset \mathbb{C}^3.$ Let $\mathbb{A}$ be the real 7-torus constructed as $\mathbb{A} := \left(\mathbb{C}^3\times\mathbb{R}\right)/\left(\Lambda\times\mathbb{Z}\right).$ We use $z_1,z_2, z_3$ as coordinates on $\mathbb{C}^3$ and $x$ as a coordinate on $\mathbb{R},$ and we notate $z_k = s_k+it_k.$ The maps $$\alpha(z_1, z_2, z_3, x) := \left(\xi z_1, \xi z_2, \xi z_3, x+\frac13\right)$$ and $$\beta(z_1, z_2, z_3, x) := \left(-\overline{z_1}, -\overline{z_2}, -\overline{z_3}, -x\right)$$ descend to diffeomorphisms of $\mathbb{A}.$ Let $\Gamma$ be the finite group generated by $\alpha$ and $\beta$ (so $\Gamma$ is isomorphic to the symmetric group on 3 letters.) There are 6 disjoint 3-dimensional tori $T_1,\dots, T_6\subset \mathbb{A}$, all with trivial normal bundle, which are fixed by some element of $\Gamma$ and permuted by the remaining elements. They are: 

\medskip
\bgroup
\def\arraystretch{1.5}
\begin{tabular}{ c | c }
    $T_1$ & $x=0, z=it$ \\
    \hline
    $T_2$ & $x=\frac12, z=it$ \\
    \hline
    $T_3$ & $x=\frac13, z=e^{i\pi/6}t$ \\
    \hline
    $T_4$ & $x=\frac56, z=e^{i\pi/6}t$ \\
    \hline
    $T_5$ & $x=\frac23, z=e^{5i\pi/6}t$ \\
    \hline
    $T_6$ & $x=\frac16, z=e^{5i\pi/6}t$ \\
\end{tabular}
\egroup
\medskip

Let $Bl_{pt}(B^4)$ be the (complex) blow-up of the open complex manifold $B^4\subset \mathbb{C}^2$ at the origin. We construct $M$ by replacing a copy of $B^4\times T_i\subset\mathbb{A}$ with $Bl_{pt}(B^4)\times T_i,$ for each $i\in\{1,\dots, 6\},$ and then taking the quotient by the induced action of $\Gamma.$ 

Note that $SL(3, \mathbb{Z})\times(\mathbb{Z}/2\mathbb{Z})$ acts on $\mathbb{A}$ by diffeomorphisms: $$(T, t)\cdot (z, x) := (Tz, x+\frac12 t),$$ and this action commutes with $\Gamma.$ These diffeomorphisms preserve the set of tori $T_i$ (although they may act nontrivially on each torus.)

\subsection{The diffeomorphisms}

Imitating \cite[Section 2]{FL}, we will construct an injection from $SL(3, \mathbb{Z})\times(\mathbb{Z}/2\mathbb{Z})$ to the mapping class group $\text{Mod}(M)$, and moreover construct explicit representatives of each mapping class in the image. (In \cite{FL} the injection is from $SL(4,\mathbb{Z})$ but the technique is similar.) So, let $(T, t)\in SL(3, \mathbb{Z})\times(\mathbb{Z}/2\mathbb{Z}).$ Our goal is to construct a diffeomorphism $f_{(T,t)}:M\to M.$ The $f_i$ in Theorem 1 will be a subset of these $f_{(T, t)}.$ For the remainder of this paper, we will take $t = 0$ for simplicity and omit it from the notation; the whole construction continues to work, though, when $t = 1.$

\subsubsection{Normal bundles and tubular neighborhoods of $T_i$}
We will use $U\subseteq \mathbb{C}^2$ to denote an open ball centered around the origin, of fixed radius which is sufficiently small. We will denote by $\nu(T_i)$ a tubular neighborhood of $T_i.$ For $i=1,2,$ let $\tau_i: U^{\subseteq\mathbb{C}^2}\times T^3\to \nu(T_i)$ be a coordinate system in a tubular neighborhood of $T_i$ that identifies $T^3 = \mathbb{R}^3/\mathbb{Z}^3$ with $T_1$ in the obvious way and sends $(a_1+ib_1,a_2+ib_2, p)\in \mathbb{C}^2\times T^3$ to: $$(a_1\partial_{s_1}+ b_1\partial_{s_2}+a_2\partial_{s_3}+b_2\partial_x, p)\in\nu(T_i).$$

For $i=3,4,5,6$ let $\tau_i:U^{\subseteq\mathbb{C}^2}\times T^3\to \nu(T_i)$ be the corresponding coordinates such that $\alpha\circ\tau_i=\tau_{i-2}$. We may choose $U$ so small that the images of the $\tau_i$ are disjoint. Observe that whenever $\beta(T_i)=T_j,$ we have $\beta\circ\tau_i=-\tau_j,$ and also observe that for each $i$ we have $T\circ \tau_i=\tau_i\circ T.$ These coordinates induce complex structures on the normal bundles of the $T_i$ which are evidently preserved by $\Gamma.$

\subsubsection{Real oriented blow-ups}
First denote by $Y$ the effect of performing a real oriented blow-up of $\mathbb{A}$ on each of the six tori $T_i.$ (So, $Y$ is a compact 7-manifold with 6 boundary components each of which is diffeomorphic to $T^3\times S^3$.) The diffeomorphism $(z, x) \mapsto(Tz, x+\frac12 t)$ extends to a diffeomorphism $f_Y:Y\to Y$ of the (real) blown-up manifold.

Define $Y'$ to be the manifold composed of $Y$ with a collar neighborhood $[0,1]\times S^3\times T^3$ glued to each boundary component $\partial\nu(T_i)$ (note that each boundary component may be identified with $S^3\times T^3$ by $\tau_i$.) Choose a path $\gamma:[0,1]\to SL(3, \mathbb{R})$ which is constant near both endpoints and such that $\gamma(0) = T\in SL(3, \mathbb{Z})$ and $\gamma(1)=\text{id}\in SL(3, \mathbb{R}).$ We may specify further that the product of the eigenvalues of $\gamma(t)$ that are greater than 1 is a non-increasing function of $t.$ Now, define the diffeomorphism $f_{Y'}:Y'\to Y'$ as follows. We let $f_{Y'}(y)=f_Y(y)$ whenever $y\in Y,$ and define $f_{Y'}$ on each copy of $[0,1]\times S^3\times T^3$ to be, for $(r, \sigma, p)\in[0,1]\times S^3\times T^3$: $$f_{Y'}(r, \sigma, p)=(r, \left(\gamma(r)\oplus\text{id}\right)(\sigma), Tp)$$

Let $\tilde{M}$ be the result of performing the \emph{Hopf collapse} as in \cite[pp. 7-8]{FL} on each blown-up torus $\{1\}\times S^3\times T^3$: via the trivializations $\tau_i,$ each $S^3\times p\subset S^3\times T^3\subset\partial Y'$ may be identified as the space of oriented real lines in the vector space $\mathbb{C}^2.$ I.e., $\tilde{M}$ ost collapse, for $\ell_1,\ell_2\in S^3\times p,$ we identify $[\ell_1]\sim[\ell_2]$ if $\ell_1$ and $\ell_2$ are both subsets of the same complex line. Since $T$ commutes with the trivializations $\tau_i$, the diffeomorphism $f_{Y'}$ descends to a diffeomorphism $f_{\tilde{M}}$ of $\tilde{M}$. 

We now must define an action of $\Gamma$ on $\tilde{M}$. Let $\Gamma$ act on $Y$ as the real oriented blow-up of the action of $\Gamma$ on $\mathbb{A}.$ This extends to an action on $Y'$ which preserves the complex structures on $NT_i$ by our computation in the previous subsection; hence, this descends to an action on $\tilde{M}.$ Moreover, this action clearly commutes with $f_{\tilde{M}}.$ Thus, $f_{\tilde{M}}$ descends to a diffeomorphism $f_T:M\to M.$ 

\section{Proof of (1) from Theorem 1.1}

The main tool in the proof of (1) from Theorem 1.1, much like in \cite{FL}, is Yomdin's proof of the Shub Entropy Conjecture for $C^\infty$ diffeomorphisms \cite{Yomdin}. This tells us that the topological entropy of a $C^\infty$ diffeomorphism $f$ of a manifold $M$ is bounded below by the log of the spectral radius of the induced action of $f$ on the homology of $M.$ We will compute this spectral radius for our diffeomorphisms $f_T$ and then compute their topological entropy; the proof will be complete when we observe that for infinitely many choices of $T,$ the topological entropy is equal to the logarithm of this spectral radius, and moreover this quantity is nonzero.

\subsection{Computing the Spectral Radius}

\begin{lemma}
    The spectral radius of $(f_T)_*$ on $H_*(M)$ is $|\lambda|^2$ where $\lambda$ is the largest eigenvalue of $T.$ Moreover, it is achieved on $H_3(M).$
\end{lemma}

\begin{proof}
    $M$ is simply connected, so by Poincar\'e duality it is enough to consider $H_2(M)$ and $H_3(M).$ The second and third homology of $M$ is the direct sum of the elements of homology coming from $\mathbb{A}/\Gamma$ and from the resolutions. The action of $f_T$ on the homology permutes the factors coming from the resolutions, so we may consider the action of $f_T$ on $H_*(\mathbb{A}/\Gamma),$ which is the same as the induced action of the linear map $T$ on $H_*(\mathbb{A}/\Gamma).$

    We will compute using cohomology for ease: the cohomology of $\mathbb{A}/\Gamma$ is the $\Gamma$-invariant part of the cohomology of $\mathbb{A}.$ Hence, $H^2(\mathbb{A}/\Gamma)$ is generated by $\text{Re}(dz_i\wedge d\overline{z_j}),$ for $i\neq j,$ and $i,j\in\{1,2,3\}$. Likewise, $H^3(\mathbb{A}/\Gamma)$ is generated by $\text{Im}(dz_1\wedge dz_2\wedge dz_3)$, as well as $\text{Im}(dz_i\wedge d\overline{z_j}\wedge dx)$, for $i,j\in\{1,2,3\}$ potentially equal.

    Letting $\lambda_1,\lambda_2, \lambda_3$ be the three eigenvalues of $T$ in non-increasing order of their modulus, the spectral radius of $(f_T)_*$ is the maximum of $|\lambda_i\lambda_j|,$ potentially with $i=j,$ and $|\lambda_i\lambda_j\lambda_k|,$ where $i, j, k$ are all different. Since $T\in SL(3, \mathbb{Z}),$ we have that $\lambda_i\lambda_j\lambda_k=1$ and the maximum of $|\lambda_i\lambda_j|,$ potentially with $i=j,$ is given by $|\lambda_1|^2.$   
\end{proof}

\subsection{Computing the Topological Entropy}


This section is devoted to computing the topological entropy $h_{top}(f_T)$ of our diffeomorphisms $f_T.$ The main lemma we will be using is as follows:

\begin{lemma}[{\cite[Propositions 3.2, 3.5]{FL}}]\label{quotients}
    Suppose the following diagram commutes, where $M$ and $N$ are compact topological spaces, $p:M\to N$ is a surjective continuous map, and $f_M$ and $f_N$ are homeomorphisms: \[\begin{tikzcd}
	M & M \\
	N & N
	\arrow["{f_M}", from=1-1, to=1-2]
	\arrow["p", from=1-1, to=2-1]
	\arrow["p", from=1-2, to=2-2]
	\arrow["{f_N}", from=2-1, to=2-2]
    \end{tikzcd}\]
    Then, $$h_{top}(f_N)\leq h_{top}(f_M)\leq h_{top}(f_N)+\sup_{n\in N}h_{top}(f_M, p^{-1}(n)).$$
\end{lemma}

And the main result of this section is:

\begin{theorem}
    The map $f_T$ has topological entropy: $$h_{top}(f_T)=2\sum_{\substack{\lambda\in\operatorname{Spec}T \\ {|\lambda|>1}}}\log |\lambda|.$$
\end{theorem}

\begin{remark}
    We imitate very closely \cite[Section 3.2]{FL}.
\end{remark}

\begin{proof}

Following \cite[Section 3.2]{FL}, we note that the linear transformation $(T, t)$ of $\mathbb{A}$ has entropy $$h_0=2\sum_{|\lambda|>1}\log |\lambda|.$$ Now, $h_Y\geq h_0$ by the quotient map property of topological entropy. By Lemma \ref{quotients}, $$h_Y\leq h_{top}(T)+\sup_{a\in\mathbb{A}} h_{top}(f_Y, \pi^{-1}(a)),$$ where $\pi:Y\to\mathbb{A}$ is the projection. Choose $a \in \mathbb{A}$ located on one of the tori $T_i,$ such that $\pi^{-1}(a)$ is a 3-sphere $S^3.$

On a normal bundle of a fixed torus $NT_i\cong T^3\times \mathbb{C}^2,$ the transformation $T$ acts by $T\in SL(3,\mathbb{Z})$ on the torus $T^3$ and by $T\otimes\text{id}$ on $\mathbb{C}^2\cong\mathbb{R}^4.$ Hence, on a boundary component $S^3\times T^3\subset\partial Y,$ the action of the diffeomorphism $f_Y$ splits according to this product: it acts by $T\in SL(3,\mathbb{Z})$ on the torus $T^3$ and by $T\otimes\text{id}$ on the space $S^3$ of rays in $\mathbb{R}^4.$ By \cite[Proposition 3.6]{FL}, the topological entropy of a linear transformation of $\mathbb{R}^n$ acting on the space of rays $S^{n-1}$ is zero. Since the subset $\pi^{-1}(a)$ projects onto a single point of the torus $T^3$ in this product, the relative entropy $h_{top}(f_Y, \pi^{-1}(a))$ must equal the entropy of $T\otimes\text{id}$ on the space $S^3$ of rays in $\mathbb{R}^4,$ so it must vanish.

Next, we compute $h_{Y'}.$ Since $f_{Y'}$ leaves invariant $Y\subset Y'$ and also each closed collar neighborhood $[0,1]\times S^3\times T^3$ of each boundary component of $Y',$ we have that: $$h_{Y'} = \max( h_Y, h_{[0,1]\times S^3\times T^3}).$$ Applying Lemma \ref{quotients} to the map $[0,1]\times S^3\times T^3\to[0,1],$ we see that $$h_{[0,1]\times S^3\times T^3}\leq \sup_{t\in [0,1]} h_{top}(\gamma(t):T^3\to T^3)=\frac12 h_0,$$ hence $h_{Y'} = h_Y = h_0.$

To compute $h_{\tilde{M}},$ note that $\tilde{M}$ is composed of the $f_{\tilde{M}}$-invariant subsets which are $Y$, and quotients of the collars $[0,1]\times S^3\times T^3$ by the Hopf collapse which is locally modeled on the Hopf map: $\{1\}\times S^3\times T^3\to \{1\}\times \mathbb{CP}^1\times T^3.$ Hence, the first inequality in Lemma \ref{quotients} tells us that the entropy of $f_{\tilde{M}}$ restricted to the quotients of the collars is bounded above by $\frac12 h_0.$ Since the entropy of $f_{\tilde{M}}|_Y$ is $h_0\geq \frac12 h_0,$ we must have that $h_{\tilde{M}} = h_0.$ Since $\tilde{M}\to M$ is a finite map, $h_{\tilde{M}}=h_{top}(f_T),$ therefore $h_{top}(f_T) = h_0$ as desired.
\end{proof}

\begin{proof}[Proof of (1) from Theorem 1.1]
From the computation of the spectral radius and topological entropy of $f_T,$ it follows that whenever there is exactly 1 eigenvalue of $T$ that is greater than 1 in magnitude, $$h_{top}(f_T) = \log\text{spec~rad}(f_T)_*,$$ so by Yomdin's proof of the Shub Entropy Conjecture for $C^\infty$ diffeomorphisms, we have that $f_T$ is a diffeomorphism that minimizes topological entropy in its homology class. There are clearly infinitely many such matrices in $SL(3, \mathbb{Z})$, with infinitely many different spectral radii, thereby sufficing for the proof of Theorem 1.
\end{proof} 

\section{Global Deformations of the $G_2$-structure on $M$}

In general, it is an interesting and open question to determine the topology of the Teichm\"uller space $\mathcal{T}(M)$. Starting with a $G_2$-structure $\widetilde\varphi(t),$ $t>0,$ on $M$ constructed by Joyce \cite[Example 7]{Joyce2}, the pullbacks $f_i^*\widetilde\varphi(t)$, for each $f_i$ in Theorem 1.1, represent different elements of the space $\mathcal{T}(M),$ and we may ask if they are in the same connected component of this space. This section will answer this question in the affirmative by proving Part (3) of Theorem 1.1.

\subsection{Joyce's construction}
We will briefly review Joyce's construction of a 1-parameter family of $G_2$ metrics on $M$. Specify $r$ to be the radius of the neighborhoods $\nu_i(T_i).$ Define, as above, $Y := \mathbb{A}/\Gamma - \nu_i(T_i),$ and let $C_1, C_2\subset M$ be the cylinders obtained by deleting $T_i$ from $\nu_i(T_i),$ so $C_1, C_2$ are isometric to a punctured neighborhood of $(\mathbb{C}^2 - \{0\})/\{\pm 1\}\times T^3$ and conformally diffeomorphic to $(0, 1]\times \mathbb{RP}^3\times T^3.$

Also, let $\Omega^+_M$ denote the subset of $\Omega^3_M$ consisting of forms which are pointwise isomorphic to the standard $G_2$-form. We define a map $\Theta:\Omega^+_M\to \Omega^4_M$ which takes $\varphi$ to $*\varphi,$ where the Hodge star is given with respect to the metric defined by the $G_2$-structure $\varphi.$ A $G_2$-structure $\varphi$ being torsion free is equivalent to $d\varphi = d\Theta(\varphi)=0.$

For $t>0$ chosen sufficiently small, Joyce constructs a torsion-free $G_2$-structure $\widetilde\varphi(t)$ on $M$, and we review this construction. He first defines $\varphi(t)$ as follows. In $M - \bigcup_i T_i,$ he defines $$\varphi'(t) := (dx_1\wedge dy_1+dx_2\wedge dy_2+dx_3\wedge dy_3)\wedge dx+\text{Im}(dz_1\wedge dz_2\wedge dz_3),$$ and he defines a 4-form $v'(t) = *\varphi'(t).$ 

On the cylinders $C_i,$ in the dual coordinates $\delta_1, \delta_2, \delta_3$ on $T^3$ and $z_1, z_2$ on $\mathbb{C}^2 - \{0\}$, this form has the expression $$\varphi'(t) = \sum_{j=1}^3\hat\omega_j\wedge\delta_j + \delta_1\wedge\delta_2\wedge\delta_3,$$ where $\hat{\omega_j}, j=1,2,3$ is a hyper-K\"ahler triple of the Euclidean metric on $\mathbb{C}^2.$

Let $X$ be the (complex) blow-up of $\mathbb{C}^2/\{\pm 1\}$ at the origin, with desingularizing map $\phi:X\to \mathbb{C}^2/\{\pm 1\}.$ Let $X_r = \phi^{-1}(B_r(0)/\{\pm1\}).$ Then, there is a surjective, proper map $\phi_i:X_r\times T^3\to C_i$ which composes $\phi\times \text{id}_{T^3}$ with the isometric coordinate chart $B_r(0)-\{0\}/\{\pm1\}\times T^3\to C_i.$ The manifold $M$, we recall, is constructed by gluing a copy of $X_r\times T^3$ to each $C_i$ via the map $\phi_i.$ Let $B_i = C_i\cup_{\phi_i} X_r.$ Define forms $\omega_i$ on $X$ as follows. Let $\omega_2, \omega_3$ be defined such that $\omega_2+i\omega_3 = dz_1\wedge dz_2$. Define a non-increasing cut-off function $\tau:[0,r^2]\to [0,1]$ such that $\tau(u) = 1$ when $u\leq r^2/4$ and $\tau(u) = 0$ when $u\geq r^2/2.$ Then, where $u$ is the squared radius function on $X_r,$ we define, for each $t>0,$ a function $f_t:X_r\to\mathbb{C}$ as follows: $$f_t := \sqrt{u^2+\tau^2(u)t^4} + \tau(u)t^2\log\left(\frac{u}{\sqrt{u^2+\tau^2(u)t^4}+\tau(u)t^2}\right).$$ Let $\omega_1(t) := \frac12 i\partial\overline{\partial}f_t.$ Then, for $u>r^2/2$, $\omega_j =\hat\omega_j,$ and for $u<r^2/4,$ the triple $\omega_j(t), j=1,2,3$ is the hyper-K\"ahler triple of the Eguchi-Hanson space. Now, define on $X_r\times T^3$: $$\varphi''(t) := \sum_{j=1}^3\omega_j(t)\wedge\delta_j + \delta_1\wedge\delta_2\wedge\delta_3, ~ v''(t) = \sum_{j=1}^3\omega_j(t)\wedge(*_{T^3}\delta_j)+\frac12 \omega_1(t)\wedge\omega_1(t).$$

The form $\varphi(t)\in\Omega^3_M$ defined by $\varphi'$ on $Y$ and $\varphi''$ on $B_i\cong X_r\times T^3$ is well-defined since they agree near the boundaries of $B_i$ and $Y.$ The form $v(t)$ defined analogously is a 4-form that approximates $\Theta(\varphi(t)).$ Joyce constructs his torsion-free $G_2$ structure $\widetilde\varphi(t)$ as a perturbation of $\varphi(t)$.

\subsection{A path of closed forms}

Let $T\in SL(3, \mathbb{Z})$ and let $f: M\to M$ be the diffeomorphism associated to $T$ by Section 2. We wish to describe $f^*\widetilde\varphi(t)$ as a gluing construction. First, for each sufficiently small $t>0,$ we will define a path $\varphi_s(t):[0,1]_s\to \Omega^3_M$ such that $\varphi_0(t) = \varphi(t)$ and $\varphi_1(t) = f^*\varphi(t).$ Like Joyce, we will also need to define a path $v_s(t):[0,1]_s\to\Omega^4_M$ which will approximate $\Theta(\varphi_s(t))$.

Note that on $Y$ we have that $f^*\varphi_1(t) = T^*\varphi(t).$ For each fixed $s\in[0,1],$ the matrix $\gamma(1-s)\in SL(3;\mathbb{R})$ does not generally descend to a well-defined map $\mathbb{A}/\Gamma\to\mathbb{A}/\Gamma.$ However, if $\alpha \in \Omega^p_\mathbb{A}$ is a $\Gamma$-invariant form whose lift $\widetilde{\alpha}$ to the universal cover of $\mathbb{A}$ is translation invariant, then $\gamma(1-s)^*\widetilde{\alpha}$ is also translation invariant, so it descends to a form $\gamma(1-s)^*\alpha\in \Omega^p_\mathbb{A},$ which is $\Gamma$-invariant because $\Gamma$ commutes with $\gamma(1-s).$ In particular, the forms: $$\varphi_{s}'(t) := {\left(\gamma(1-s)\right)^*}(\varphi(t)),~ v_s'(t) := {\left(\gamma(1-s)\right)^*}(v(t))$$ on $\mathbb{A}/\Gamma$ are well-defined, as they come from translation-invariant forms on the universal cover of $\mathbb{A}.$ Moreover, the forms $\varphi_s'(t)$ and $v_s'(t)$ are closed and satisfy $\Theta(\varphi_s'(t))=v_s'(t),$ since these are local statements and they hold true on the universal cover of $\mathbb{A}$.

Recall that each $T\in SL(3; \mathbb{R})$ acts on $\mathbb{RP}^3$ by $T\oplus \text{id}_\mathbb{R}$ acting on lines in $\mathbb{R}^4.$ On each collar $C_i\cong (B_r(0)/\{\pm1\})\times T^3,$ the form $\varphi_{s}'(t)$ has the expression: $$\sum_{j=1}^3\gamma(1-s)^*\hat\omega_j(t)\wedge\gamma(1-s)^*(\delta_j) + \gamma(1-s)^*(\delta_1\wedge\delta_2\wedge\delta_3).$$ Define a non-decreasing cut-off function $\sigma:[0,r]\to[0,1]$ such that if $\rho^2<2r^2/3$ then $\sigma(\rho)=0$ and if $\rho^2>3r^2/4$ then $\sigma(\rho)=1.$ For each $s\in[0,1]$, let $\Psi_s:(B_r(0)-\{0\})/\{\pm1\}\to (B_r(0)-\{0\})/\{\pm1\}$ be defined in polar coordinates $(0,r]\times S^3/\{\pm1\}$ as $\Psi_s(\rho, \alpha) = (\rho, \gamma(1-s\sigma(\rho))\alpha).$ Then, upon applying an isotopy to $f$ if necessary, the form $f^*\varphi'$ has the expression on $C_i$: $$f^*\varphi'=\sum_{j=1}^3\Psi_1^*\hat\omega_j(t)\wedge T^*(\delta_j) + T^*(\delta_1\wedge\delta_2\wedge\delta_3).$$ Define $\varphi''_s(t)$ on $X_r\times T^3$ to be: $$\varphi_s''(t):=\sum_{j=1}^3\Psi_s^*\omega_j(t)\wedge \gamma(1-s)^*(\delta_j) + \gamma(1-s)^*(\delta_1\wedge\delta_2\wedge\delta_3).$$ Similarly, define $v_s''(t)$ on $X_r\times T^3$ to be: $$v_s''(t) = \sum_{j=1}^3 \Psi_s^*\omega_j(t)\wedge\gamma(1-s)^*(*_{T^3}\delta_j)+\frac12 \Psi_s^*\omega_1(t)\wedge\Psi_s^*\omega_1(t).$$For each $s\in[0,1]$ and $t>0$ sufficiently small, we now define $\varphi_s(t)\in \Omega^3_M$ as before, to be defined by $\varphi'_s$ on $Y$ and $\varphi''_s$ on each $B_i\cong X_r\times T^3$, and similarly define $v_s(t)$. Again, they are well-defined since $\varphi_s'(t)$ and $\varphi_s''(t)$ (respectively $v_s'(t)$ and $v_s''(t)$) agree near the boundaries of $B_i$ and $Y.$ Moreover, it is clear from the local expressions that $\varphi_s(t)$ is closed for each fixed $s, t,$ and is a $G_2$-structure whenever $t$ is sufficiently small. Indeed, this latter claim amounts to checking in the local expressions on $X_r\times T^3$ that the quadratic form $(v, w)\mapsto \iota_v\varphi_s(t)\wedge\iota_w\varphi_s(t)\wedge\varphi_s(t)$ is non-degenerate, which follows easily once we observe that under this form the splitting $X_r\times T^3$ is orthogonal.

Define $\psi_s(t) = *\left(\Theta(\varphi_s(t)) - v_s(t)\right)$ where we take the Hodge star relative to the metric on $M$ defined by $\varphi_s(t).$

\subsection{Gluing in families}

This section will produce a family $\widetilde{\varphi}_s(t)$ of torsion-free $G_2$-structures which is well-approximated by $\varphi_s(t).$ Our main tool will be a families version of Joyce's gluing results \cite[Theorems 11.6.1, G1, G2]{Joycebook} due to Crowley, Goette, and Hertl, which we state only in its untwisted case:

\begin{theorem}[special case of {\cite[Theorem 6.2]{CGH}}]
    Suppose that for each sufficiently small $t>0,$ the form $\varphi_s(t) \in \Omega^3_M$ is a continuous family of not necessarily torsion-free $G_2$ structure parameterized by a compact space $B, s\in B.$ Let $g_s(t)$ be the corresponding Riemannian metric. Suppose moreover that there exists $\psi_s(t) \in \Omega^3_M$ also depending continuously on $s\in B$, such that $d^*\psi_s(t) = d\Theta(\varphi_s(t)),$ and there exist constants $A_1, A_2,$ and $A_3$ such that the following 5 estimates hold independently of $s\in B$: \begin{enumerate}
        \item $||\psi_s(t)||_{L^2}\leq A_1 t^4$
        \item $||\psi_s(t)||_{C^0}\leq A_1 t^{1/2}$
        \item $||d^*\psi_s(t)||_{L^{14}}\leq A_1$
        \item The injectivity radius of $g_s(t)$ is bounded below by $A_2t.$
        \item The $C^0$-norm of the Riemann curvature of $g_s(t)$ is bounded above by $A_3t^{-2}.$ 
    \end{enumerate}
    Then, for all sufficiently small $t$, there exists a constant $K$ and $\widetilde{\varphi}_s(t)$ a torsion-free $G_2$-structure satisfying the following: \begin{enumerate}
        \item $\widetilde{\varphi}_s(t)$ depends continuously on $s.$
        \item $\widetilde{\varphi}_s(t)$ is cohomologous to $\varphi_s(t).$
        \item $||\widetilde{\varphi}_s(t) - \varphi_s(t)||_{C^0}\leq Kt^{1/2}.$
    \end{enumerate}
\end{theorem}

\begin{remark}
    That $\widetilde{\varphi}_s(t)$ is cohomologous to $\varphi_s(t)$ is not stated explicitly in \cite{CGH} but follows immediately from their construction of $\widetilde{\varphi}_s(t)$ as $\varphi_s(t)+d\eta_s(t)$ for a particular choice of $\eta_s(t)\in\Omega^2_M.$
\end{remark}

We are now ready to prove the main theorem of this section:

\begin{theorem}
    For all sufficiently small $t>0,$ there exists a continuous path of torsion-free $G_2$ structures $\widetilde{\varphi}_s(t)\in\Omega^3_M, s\in[0,1]$ such that $\widetilde{\varphi}_0(t) = \widetilde{\varphi}(t)$ and $\widetilde{\varphi}_1(t) = f^*\widetilde{\varphi}(t).$
\end{theorem}

\begin{proof}
    We will first apply the families gluing theorem to our family $\varphi_s(t)$ constructed in the previous section. To do so, we must prove the five estimates in that theorem. Note that $\psi_s(t)$ is supported entirely in the annulus $\frac{r^2}4\leq\rho^2\leq\frac{3r^2}{4}$ within $X_r\times T^3.$ In this annulus, $\psi_0(t)$ is translation-invariant in the $T^3$ direction, so we may express: $\psi_s(t) = (\Psi_s^*\oplus\gamma(1-s)^*)(\psi_0(t))$ and the metric is obtained by applying $(\Psi_s^*\oplus\gamma(1-s)^*)$ to the metric at $s = 0.$ By \cite[Theorem 11.5.7]{Joycebook}, bounds (1)-(5) hold for $\psi_0(t).$ Since $\Psi_s$ and its derivatives, and $\gamma(1-s)^*$ are all bounded independently of $s$, as are their inverses, we must have bounds (1)-(5) in our families situation. The families gluing theorem therefore gives us $\widetilde{\varphi}_s(t)$ torsion-free. Moreover, $\widetilde{\varphi}_0(t) = \widetilde{\varphi}(t)$ since they are constructed in the identical way, according to the proofs of \cite[Theorem 6.2]{CGH} and \cite[Theorem 11.6.1]{Joycebook}.

    What's left to prove is that $\widetilde{\varphi}_1(t) = f^*\widetilde\varphi(t).$ Since evidently, $\psi_1(t) = f^*\psi_0(t)$, and the iteration procedure used to construct $\widetilde{\varphi}(t)$ in \cite[Theorem 11.6.1]{Joycebook} and \cite[Theorem 6.2]{CGH} is diffeomorphism invariant, we must have that $\widetilde{\varphi}_1(t) = f^*\widetilde\varphi(t),$ as desired.
\end{proof}

\section{Dynamics on $G_2$ Manifolds}

It is tempting to place these examples in a general theory of dynamical systems on $G_2$ manifolds, which may be analogous to dynamics on K\"ahler manifolds. In the latter theory, biholomorphic automorphisms often comprise a large, interesting class of dynamical systems to study. However, while many similarities exist between the cohomology of $G_2$ manifolds and that of K\"ahler manifolds, to the author's knowledge there is no natural $G_2$ analog of a biholomorphic diffeomorphism. The purpose of this section is to illustrate the failure of one such potential analog from producing interesting dynamics.

\subsection{A local formula}

Let $V = \mathbb{R}^7$ be the standard representation of $G_2$, with $\varphi_0\in\Lambda^3 V^*$ stabilized by $G_2$ and a $G_2$-invariant metric $\langle, \rangle$ (we normalize $\varphi_0$ such that $|\varphi_0|^2 = 7.$) Following Verbitsky in \cite{Verbitsky}, we define an operator $C:\Lambda^1V^*\to\Lambda^2V^*$ by $\alpha\mapsto \varphi_0(\alpha^\#, \cdot, \cdot),$ where $\#:V^*\to V$ is the musical isomorphism of the metric $\langle, \rangle.$ We may extend $C$ to a map $C:\Lambda^kV^*\to\Lambda^{k+1}V^*$ by enforcing the graded Leibniz rule.

\subsection{Harmonic forms on a $G_2$-manifold} 

The representation $\Lambda^3V^*$ of $G_2$ decomposes into irreducibles as follows, where the subscript indicates the dimension (see, for instance, \cite{Bryant} or \cite{Joyce1}): $$\Lambda^3V^* = \Lambda^3_1\oplus\Lambda^3_7\oplus\Lambda^3_{21},$$ where $\Lambda^3_1=\mathbb{R}\langle\varphi_0\rangle$ and $\Lambda^3_{21}$ consists of those forms $\alpha$ such that $$\alpha\wedge\varphi_0=\langle\alpha,\varphi_0\rangle = 0.$$ The Verbitsky operator $C$ splits according to this decomposition \cite[Claim 3.11, Remark 3.12]{Verbitsky}: on $\Lambda^3_1$ it acts by $C\varphi_0 = 2*\varphi_0,$ while on $\Lambda^3_{27}$ it acts by $C\alpha = -*\alpha.$ 

On a $G_2$-manifold $(M, \varphi),$ this splitting globalizes to a splitting of $\Omega^3_M=\Omega^3_1\oplus\Omega^3_7\oplus\Omega^3_{21}$ which in turn gives rise to a splitting of the harmonic forms $\mathcal{H}^3_M = \mathbb{R}\langle\varphi\rangle\oplus\mathcal{H}^3_7\oplus\mathcal{H}^3_{21}$. If $M$ is irreducible, then $\mathcal{H}^3_7 = 0$ \cite[Theorem 10.2.4]{Joycebook}. The map $C$ globalizes to a map $C:\Omega^3_M\to\Omega^{4}_M$ that preserves the subspace of harmonic forms. The local relationship between $C$ and the Hodge star operator immediately gives us the following proposition:

\begin{proposition}
    If $(M, \varphi)$ is a closed irreducible $G_2$-manifold, then the $G_2$ structure $\varphi$ determines a quadratic form $Q:H^3(M; \mathbb{R})\times H^3(M; \mathbb{R})\to\mathbb{R}$ with signature $(1, b_3-1)$ defined on the space of harmonic 3-forms by: $$Q(\alpha,\beta) = \int_M\alpha\wedge C(\beta).$$
\end{proposition}

\subsection{Diffeomorphisms of $G_2$-Manifolds preserving $Q$} The form $Q$ is reminiscent of the intersection form $Q_X$ on $H^{1,1}(X),$ where $X$ is K3 surface, given that both forms have signature $(1, n-1).$ Biholomorphic automorphisms of $X$ preserve $Q_X$, and therefore act by hyperbolic isometries on each sheet of the hyperboloid $\{\alpha\in H^{1,1}(X) | Q_X(\alpha, \alpha) = 1\}$. Cantat classified biholomorphic automorphisms according to how they act on this hyperbolic space \cite{Cantat}. We would hope for an analogous classification of $Q$-preserving automorphisms of $M.$ Below we propose a dictionary between dynamically interesting features of $K3$ surfaces (or, in fact, algebraic surfaces in general) and those of irreducible $G_2$ manifolds:

\medskip
\bgroup
\def\arraystretch{1.25}
\begin{tabular}{ c | c }
$\mathbf{K3}$ \textbf{Surface} $(X, \omega)$ & \textbf{Irreducible} $\mathbf{G_2}$ \textbf{Manifold} $(M, \varphi)$ \\
\hline
K\"{a}hler form $\omega$ & $G_2$-form $\varphi$ \\
\hline
$H^{1,1}(X)$ & $H^3(M)$ \\
\hline
Intersection form on $H^{1,1}(X)$ & The form $Q$ on $H^3(M)$ \\
\hline
Complex structure & Verbitsky operator $C$ \\
\hline
Moduli space of marked K3's & Teichm\"uller space $\mathcal{T}(M)$ \\
\end{tabular}
\egroup
\medskip

This next proposition shows exactly which diffeomorphisms of $M$ preserve $Q$ locally. 

\begin{proposition}
    Let $(M,\varphi)$ be a closed, irreducible $G_2$ manifold. Suppose the diffeomorphism $f:M\to M$ satisfies $f^*\alpha\wedge C(f^*\beta) = f^*(\alpha\wedge C(\beta))$ for all $\alpha, \beta\in\Omega^3_1\oplus\Omega^3_{27}.$ Then, $f$ is an isometry.
\end{proposition}

\begin{proof}
    
    As in the beginning of this section, let $V = \mathbb{R}^7$ be the standard representation of $G_2.$ Let $q\in\Sym^2((\Lambda^3_1\oplus\Lambda^3_{27})^*)\otimes\Lambda^7V^*$ be the quadratic form $q(\alpha, \beta) = \alpha\wedge C\beta.$ The group $GL(V)$ acts on the space of quadratic forms $\Sym^2(\Lambda^3_1\oplus\Lambda^3_{27})\otimes\Lambda^7V^*$, and we claim that the stabilizer of $q$ under this action is a subgroup of the orientation preserving conformal group $\mathbb{R}_\times\cdot SO(7).$ In fact, we will show the equivalent statement that any matrix in $SL(V)$ preserving $q$ is in $SO(7).$ 
    
    Let $H\subseteq SL(V)$ be the subset of matrices stabilizing $q.$ The derivative of the action of $SL(V)$ on $\Sym^2(\Lambda^3_1\oplus\Lambda^3_{27})\otimes\Lambda^7V^*$ at $q$ is a linear map $$T:\mathfrak{sl}(V)\to \Sym^2(\Lambda^3_1\oplus\Lambda^3_{27})\otimes\Lambda^7V^*$$ which is $G_2$-invariant. The Lie algebra $\mathfrak{h}$ of $H$ is clearly given by $\mathfrak{h} = \ker T.$ We will first show that $\mathfrak{h}$ consists entirely of skew-symmetric matrices. 
    
    To compute this kernel, we take advantage of the $G_2$ invariance of $T.$ In the orthogonal decomposition of the vector space $\mathfrak{sl}(V)$ into traceless symmetric matrices and anti-symmetric matrices, $$\mathfrak{sl}(V) \cong \Sym^2_0V\oplus\mathfrak{so}(V),$$ the summand $\Sym^2_0V$ is irreducible as a representation of $G_2$ (see \cite{Bryant} or \cite[Chapter 10]{Joycebook}.) So, by Schur's Lemma, to show that $\Ker T\cap {\Sym^2_0V} = \{0\},$ it suffices to show that $T$ is non-vanishing at a single element, in other words, that for some $X\in \Sym^2_0V$ there exists $\eta\in\Lambda^3_1\oplus\Lambda^3_{27}$ for which $$q(X\eta, \eta)+q(\eta, X\eta)\neq0.$$ Consider the element $\eta = e_5\wedge(e_1\wedge e_2 - e_3\wedge e_4)\in\Lambda^3_{27}$ and $X\in \Sym^2_0V$ given by: $$X = \begin{pmatrix}0 & 1 & 0& 0& 0& 0& 0\\ 1 &0  &0 &0 &0 &0 &0 \\ 0& 0&1 &0 &0 &0 &0 \\ 0 & 0& 0& 0& 0& 0& 0\\  0&0 &0 &0 &0 &0 &0 \\0&0 &0 &0 &0 &0 &0 \\ 0&0 &0 &0 &0 &0 & -1\\\end{pmatrix}.$$ We compute that $X\cdot\eta = -e_3\wedge e_4\wedge e_5$ decomposes according to $\Lambda^3 V^* = \Lambda^3_1\oplus\Lambda^3_7\oplus\Lambda^3_{27}$ as: $$X\cdot \eta = -\frac17\varphi_0+0+\left(X\cdot \eta +\frac17\varphi_0\right).$$ Using the two facts that (1) $\Lambda^3_1$ and $\Lambda^3_{27}$ are orthogonal with respect to $q$ and (2) that on $\Lambda^3_{27}$, we have $q(\alpha, \beta)=-\langle\alpha, \beta\rangle \text{d}vol$ \cite[p. 25]{Verbitsky}, we compute: $$q(X\eta, \eta)+q(\eta, X\eta) = -e_1\wedge\dots\wedge e_7\neq 0.$$
    Therefore, $\Ker T\subseteq \Lambda^2V = \mathfrak{so}(V),$ which shows that $H_0,$ the connected component of the identity of $H,$ is a closed subgroup of $SO(V).$ Since $H$ is clearly real algebraic, it has finitely many connected components \cite[Theorem 3]{Whitney}. Since each component is homeomorphic to $H_0$ by left multiplication, $H$ must be compact. Therefore, $H$ must preserve some positive definite quadratic form on $V.$ We know that $G_2\subseteq H$ and $G_2$ preserves a unique such form, therefore $H$ must preserve that same form; so $H \subseteq SO(V),$ which suffices for our claim.

    Now, take $f:M\to M$ as in the statement of the proposition. By our claim, we must have that $f$ preserves the conformal structure on $M.$ However, $G_2$-metrics are Ricci-flat, so \emph{a fortiori} they are scalar-flat \cite{Bonan}. Letting $g$ be the $G_2$-metric on $M,$ since $f$ is conformal, we must have $$f^*g = e^{2u}g$$ for some function $u:M\to\mathbb{R}.$ Moreover, $f^*g$ is also scalar-flat since it is a pullback of a scalar-flat metric. The equation for the scalar curvature under conformal change in dimension 7 is: $$s(e^{2u}g) = e^{-2u}\left(s(g) + 12\Delta u-30|du|^2\right)$$ (where we use $\Delta u = d^*du.$) In our situation, this tells us $12\Delta u = 30 |du|^2.$ Integrating both sides of this equation over $M,$ we see that $du = 0,$ so $u$ is a constant. Since $f$ is a diffeomorphism, it must preserve the volume of $M$, which forces $u = 0.$ This proves that $f$ is an isometry.
\end{proof}

Unfortunately, isometries of irreducible $G_2$ manifolds are dynamically uninteresting since they are all finite order, as proved by Crowley-Goette-Hertl using a Bochner formula:

\begin{theorem}[{\cite[Lemma 2.1]{CGH}}]
    The isometry group of a closed, irreducible $G_2$-manifold is finite.
\end{theorem}

Their proof is ineffective, and this circle of ideas does raise two outstanding questions about $G_2$ isometries that would be interesting to consider:

\begin{question}[$G_2$ Hurwitz Inequality]
    Is there a topological bound on the number of isometries of a closed, irreducible $G_2$-manifold coming from, say, its Betti numbers? 
\end{question}

\begin{question}[$G_2$ Nielsen Realization]
    For any finite subgroup $\Gamma$ of the mapping class group of a closed, irreducible $G_2$-manifold $M$, does there exist a $G_2$-metric on $M$ such that $\Gamma$ is realized by isometries?
\end{question}

\bibliographystyle{alpha}
\bibliography{main}

\end{document}